\documentclass[a4paper,12pt]{article}
\usepackage{times, url}
\usepackage{amsmath,amssymb,amsthm,amsfonts}

\newtheorem{theorem}{Theorem}[section]

\newtheorem{lemma}[theorem]{Lemma}

\newtheorem{problem}{Problem}[section]
\newtheorem{example}{Example}[section]
\newtheorem{definition}{Definition}[section]

\begin{document}
\setcounter{page}{1}

\begin{center}
{\LARGE \bf  Canonical matrices with entries integers modulo $p$}
\vspace{8mm}

{\Large \bf Krasimir Yordzhev}
\vspace{3mm}

Trakia University \\
Stara Zagora, Bulgaria \\
e-mail: \url{krasimir.yordzhev@trakia-uni.bg}
\vspace{2mm}

\end{center}
\vspace{10mm}

\noindent
{\bf Abstract:} The work considers an equivalence relation in the set of all  $n\times m$ matrices with entries in the set $[p]=\{ 0,1,\ldots , p-1 \}$. In each element of the factor-set  generated by this relation, we define the concept of canonical matrix, namely the minimal element with respect to the lexicographic order.  We have found a necessary and sufficient condition for an arbitrary matrix with entries in the set $[p]$ to be canonical. For this purpose, the matrices are uniquely represented by ordered $n$-tuples of integers.\\
{\bf Keywords:} Permutation matrix, Weighing matrix, Hadamard matrix, Semi-canonical matrix, Canonical matrix, Ordered $n$-tuples of integers. \\
{\bf 2010 Mathematics Subject Classification:} 05B20, 15B36.
\vspace{5mm}

\section{Introduction  and notation}
This paper presents a generalization and an improvement of the results obtained in \cite{YORDZHEV2017122}.

Let $k$ and $p$ be integers, $k\le p$. By $[k,p]$  we denote the set $$[k,p] =\left\{ k,k+1,\ldots ,p\right\}$$  and by $[p]$  the set  $$[p] =[0,p-1] =\left\{ 0,1,2,\ldots ,p-1 \right\} .$$

With ${\mathcal M}_{n\times m}^p $ we will denote the set of all $n\times m$ matrices with entries in the set $[p]$.

When $p=2$, a matrix whose entries belong to the set $[2]=\{ 0,1\} $ is called \emph{binary} (or \emph{boolean}, or \emph{(0,1)-matrix}).

When $p=3$ a $n\times n$ matrix $H$ whose entries belong to the set $\{ 1,-1\}\equiv \{ 1,2\}\ (mod\ 3) $ is \emph{Hadamard} if $H H^T =n\, I_n$,
where $H^T$ is the transposed matrix of $H$ and $I_n$ is the $n\times n$ identity matrix. It is well known that $n$ is necessarily 1, 2, or a multiple of four \cite{Hedayat,Horadam}.

When $p=3$ a $n\times n$ matrix $W$ whose entries belong to the set $\{ 0,1,-1\} \equiv \{ 0,1,2\} \ (mod\ 3)$ is \emph{weighing matrix} of order $n$ with weight $k$, if
$W\, W^T =k\, I_n$. For more information on applications of weighing matrices, we refer the reader to \cite{KOUKOUVINOS199791}. A $n\times n$  weighing matrix $W$ with weight $k$ is Hadamard if $k=n$ (see \cite{Best2010}).

A square binary matrix is called a \textit{permutation matrix}, if there is exactly one 1 in every row and every column. Let us denote the group of all $n\times n$  permutation matrices by ${\mathcal P}_{n} $. It is well known (see \cite{Sachkov,tar2}) that the multiplication of an arbitrary real or complex matrix $A$ from the left with a permutation matrix (if the multiplication is possible) leads to permutation of the rows of the matrix $A$, while the multiplication of $A$ from the right with a permutation matrix leads to permutation of the columns of $A$.

A \emph{transposition} is a matrix obtained from the $n\times n$ identity matrix $I_n$ by interchanging two rows or two columns.
With ${\mathcal T}_{n} \subset {\mathcal P}_{n} $ we denote the set of all transpositions in ${\mathcal P}_{n} $, i.e. the set of all $n\times n$ permutation matrices, which multiplying from the left an arbitrary $n\times m$ matrix swaps the places of exactly two rows, while multiplying from the right an arbitrary $k\times n$ matrix swaps the places of exactly two columns.

\begin{definition}\label{Definition3}
Let $A,B\in {\mathcal M}_{n\times m}^p$. We will say that the matrices $A$ and $B$ are \emph{equivalent} and we will write
$$A{\rm \sim }B, $$
if there exist permutation matrices $X\in {\mathcal P }_{n} $ and $Y\in {\mathcal P}_{m} $, such that
$$A=XBY.$$
\end{definition}

\par In other words $A \sim B$, if $A$ is received from $B$ after a permutation of some of the rows and some of the columns of $B$.
Obviously, the introduced relation is an equivalence relation.

In each element of the factor-set  generated by the relation $\sim$ described in definition \ref{Definition3}, we define the concept of canonical matrix, namely the minimal element with respect to the lexicographic order. For this purpose, the matrices are uniquely represented by ordered $n$-tuples of integers. The purpose of this work is to get a necessary and sufficient condition for an arbitrary matrix with entries in the set $[p]$ to be canonical. This task is solved in the particular case where $p = 2$ in \cite{YORDZHEV2017122}.
The case where $p = 3$ will be useful in classification of Hadamard matrices and weighing matrices.

\section{Representation of matrices from ${\mathcal M}_{n\times m}^p$ via ordered n-tuples of integers}

Let $A=\left( a_{ij} \right)_{n\times m} \in {\mathcal M}_{n\times m}^p $, $1\le i\le n$, $1\le j\le m$ and let
\begin{equation}\label{x_i}
x_i =\sum_{j=1}^m a_{ij} p^{m-j} ,\ i=1,2,\ldots n.
\end{equation}

Obviously
\begin{equation}\label{p^m-1}
0\le x_{i} \le p^{m} -1\quad \textrm{for every} \quad i=1,2,\ldots n
\end{equation}
and $x_{i} $ is a natural number written in notation in the number system with the base $p$ whose digits are consistently the entries of the $i$-th row of $A$.

With $r(A)$ we will denote the ordered $n$-tuple
\begin{equation}\label{r(A)}
r(A)=\langle x_{1} ,x_{2} ,\ldots ,x_{n} \rangle .
\end{equation}

Similarly with $c(A)$ we will denote the ordered $m$-tuple
\begin{equation}\label{c(A)}
c(A)=\langle y_{1} ,y_{2} ,\ldots ,y_{m} \rangle ,
\end{equation}
where
\begin{equation}\label{y_i}
y_j =\sum_{i=1}^n a_{ij} p^{n-i},\quad 0\le y_{j} \le p^{n} -1,\quad j=1,2,\ldots m
\end{equation}
and $y_{j} $ is a natural number written in notation in the number system with the base $p$ whose digits are consistently the entries of the $i$-th column of $A$.

It is easy to see that for every  $A\in {\mathcal M}_{n\times m}^p $,  $c(A)=r(A^T )$ and $r(A)=c(A^T )$, where $A^T$ is the transposed matrix of $A$.

We consider the sets:

\[\begin{array}{lll} {{\mathcal R}_{n\times m}^p } & {=} &  [0,p^m -1]^n \\
      & =   &       {\left\{\langle x_{1} ,x_{2} ,\ldots ,x_{n} \rangle \; |\; 0\le x_{i} \le p^{m} -1,\; i=1,2,\ldots n\right\}} \\
{ }   & {=} & {\left\{r(A)\, |\; A\in {\mathcal M}_{n\times m}^p \right\}}  \end{array}\]
and

\[\begin{array}{lll} {{\mathcal C}_{n\times m}^p } & = & [0,p^n -1]^m \\
    & {=} & {\left\{\langle y_{1} ,y_{2} ,\ldots ,y_{m} \rangle \; |\; 0\le y_{j} \le p^{n} -1,\; j=1,2,\ldots m\right\}} \\
    & {=} & {\left\{c(A)\, |\; A\in {\mathcal M}_{n\times m}^p \right\}} \end{array}\]

Thus we define the following two mappings:
$$r: {\mathcal M}_{n\times m}^p \to {\mathcal R}_{n\times m}^p $$
and
$$c: {\mathcal M}_{n\times m}^p \to  {\mathcal C}_{n\times m}^p ,$$
which are bijective and therefore
$${\mathcal R}_{n\times m}^p \cong {\mathcal M}_{n\times m}^p \cong {\rm {\mathcal C}}_{n\times m}^p .$$

We will denote the lexicographic orders in ${\rm {\mathcal R}}_{n\times m}^p $ and in ${\rm {\mathcal C}}_{n\times m}^p$  with $<$.

\begin{example}\label{exexex1}
\rm Let
$$
A=\left(\begin{array}{cccc}
{1} & {0} & {3} & {2} \\
{0} & {2} & {1} & {0} \\
{0} & {1} & {1} & {3} \\
\end{array}\right) \in {\mathcal M}_{3\times 4}^4
$$
Then
$$x_1 =1\cdot 4^3 +0\cdot 4^2 +3\cdot 4^1 +2\cdot 4^0 =1\cdot 64 +0\cdot 16+3\cdot 4+2\cdot 1=78,$$
$$x_2 =0\cdot 4^3 +2\cdot 4^2 +1\cdot 4^1 +0\cdot 4^0 =0\cdot 64 +2\cdot 16+1\cdot 4+0\cdot 1=36,$$
$$x_3 =0\cdot 4^3 +1\cdot 4^2 +1\cdot 4^1 +3\cdot 4^0 =0\cdot 64 +1\cdot 16+1\cdot 4+3\cdot 1=23,$$
$$y_1 =1\cdot 4^2 +0\cdot 4^1 +0\cdot 4^0  =1\cdot 16+0\cdot 4+ 0\cdot 1=16,$$
$$y_2 =0\cdot 4^2 +2\cdot 4^1 +1\cdot 4^0  =0\cdot 16+2\cdot 4+1\cdot 1=9,$$
$$y_3 =3\cdot 4^2 +1\cdot 4^1 +1\cdot 4^0  =3\cdot 16+1\cdot 4+1\cdot 1=53,$$
$$y_4 =2\cdot 4^2 +0\cdot 4^1 +3\cdot 4^0  =2\cdot 16+0\cdot 4+3\cdot 1=35,$$
$$r(A)=\langle 78,36,23\rangle ,$$
$$c(A)= \langle 16,9,53,35\rangle .$$
\end{example}

\begin{theorem}\label{Theorem1}
Let $A$ be an arbitrary matrix from ${\mathcal M}_{n\times m}^p $. Then:

a) If $X_{1} ,X_{2} ,\cdots ,X_{s} \in {\rm {\mathcal T}}_{n} $ are such that
$$r(X_{1} X_{2} \ldots X_{s} A)<r(X_{2} X_{3} \ldots X_{s} A)<\cdots <r(X_{s-1} X_s A) <r(X_{s} A)<r(A),$$
then
$$c(X_{1} X_{2} \ldots X_{s} A)<c(A).$$

b) If $Y_{1} ,Y_{2} ,\cdots ,Y_{t} \in {\rm {\mathcal T}}_{m} $ are such that
$$c(AY_{1} Y_{2} \ldots Y_{t})<c(AY_{1} Y_{2} \ldots Y_{t-1})<\cdots <c(AY_1 Y_2 )< c(AY_1 )<c(A),$$
then
$$r(AY_{1} Y_{2} \ldots Y_{t} )<r(A).$$
\end{theorem}

\begin{proof}

a) Induction by $s$.

Let $s=1$ and let $X\in {\mathcal T}_{n} $ be a transposition which multiplying an arbitrary matrix $A=(a_{ij} )\in {\mathcal M}_{n\times m}^p $ from the left swaps the places of the rows of $A$ with numbers $u$ and $v$ ($1\le u<v\le n$), while the remaining rows stay in their places. In other words if
$$A=\left(
  \begin{array}{cccccc}
    a_{11} & a_{12} & \cdots & a_{1r} & \cdots & a_{1m} \\
    a_{21} & a_{22} & \cdots & a_{2r} & \cdots & a_{2m} \\
    \vdots & \vdots &        & \vdots &        & \vdots \\
    a_{u1} & a_{u2} & \cdots & a_{ur} & \cdots & a_{um} \\
    \vdots & \vdots &        & \vdots &        & \vdots \\
    a_{v1} & a_{v2} & \cdots & a_{vr} & \cdots & a_{vm} \\
    \vdots & \vdots &        & \vdots &        & \vdots \\
    a_{n1} & a_{n2} & \cdots & a_{nr} & \cdots & a_{nm} \\
  \end{array}
\right)
$$
then
$$X A=\left(
  \begin{array}{cccccc}
    a_{11} & a_{12} & \cdots & a_{1r} & \cdots & a_{1m} \\
    a_{21} & a_{22} & \cdots & a_{2r} & \cdots & a_{2m} \\
    \vdots & \vdots &        & \vdots &        & \vdots \\
    a_{v1} & a_{v2} & \cdots & a_{vr} & \cdots & a_{vm} \\
    \vdots & \vdots &        & \vdots &        & \vdots \\
    a_{u1} & a_{u2} & \cdots & a_{ur} & \cdots & a_{um} \\
    \vdots & \vdots &        & \vdots &        & \vdots \\
    a_{n1} & a_{n2} & \cdots & a_{nr} & \cdots & a_{nm} \\
  \end{array}
\right) ,
$$
where $a_{ij} \in [p]=\{ 0,1,\ldots ,p-1\}$, $1\le i\le n$, $1\le j\le m$.

Let $$r(A)=\langle x_{1} ,x_{2} ,\ldots x_{u-1},x_{u} ,\ldots x_{v-1},x_{v} ,\ldots ,x_{n} \rangle .$$

Then $$r(XA)=\langle x_{1} ,x_{2} ,\ldots x_{u-1},x_{v} ,\ldots x_{v-1},x_{u} ,\ldots ,x_{n} \rangle .$$

Since $r(XA)<r(A)$, then according to the properties of the lexicographic order  $x_{v} <x_{u} $. Let the representation of $x_{u} $ and $x_{v} $ in notation in the number system with the base $p$ (with an eventual addition of unessential zeros in the beginning if necessary) be respectively as follows:
$$x_{u} =a_{u1} a_{u2} \cdots a_{ur}\cdots a_{um} ,$$
$$x_{v} =a_{v1} a_{v2} \cdots a_{vr}\cdots a_{vm} .$$

Since $x_{v} <x_{u} $, then there exists an integer $r\in \{ 1,2,\ldots ,m\} $, such that $a_{uj} =a_{vj} $ when $j<r$,  and $a_{vr} <a_{ur}$.
Hence if $c(A)=\langle y_{1} ,y_{2} ,\ldots ,y_{m} \rangle $, $c(XA)=\langle z_{1} ,z_{2} ,\ldots ,z_{m} \rangle $, then $y_{j} =z_{j} $ when $j<r$, while the representation of $y_{r} $ and $z_{r} $ in notation in the number system with the base $p$ (with an eventual addition of unessential zeros in the beginning if necessary) is respectively as follows:
$$y_{r} =a_{1r} a_{2r} \cdots a_{u-1{\kern 1pt} r} a_{ur} \cdots a_{vr} \cdots a_{nr} ,$$
$$z_{r} =a_{1r} a_{2r} \cdots a_{u-1{\kern 1pt} r} a_{vr} \cdots a_{ur} \cdots a_{nr} .$$

Since $a_{vr} <a_{ur}$, then $z_{r} <y_{r} $, whence it follows that $c(XA)<c(A)$.

We assume that for every $s$-tuple of transpositions $X_{1} ,X_{2} ,\ldots ,X_{s} \in {\mathcal T}_{n} $ and for every matrix $A\in {\mathcal M}_{n\times m}^p $ from
$$r(X_{1} X_{2} \ldots X_{s} A)<r(X_{2} \cdots X_{s} A)<\cdots <r(X_{s} A)<r(A)$$
it follows that
$$c(X_{1} X_{2} \ldots X_{s} A)<c(A)$$
and let $X_{s+1} \in {\mathcal T}_{n} $ be such that
$$r(X_{1} X_{2} \ldots X_{s} X_{s+1} A)<r(X_{2} \cdots X_{s+1} A)<\cdots <r(X_{s+1} A)<r(A).$$

According to the above proved  $c(X_{s+1} A)<c(A)$.

We put

$$A_{1} =X_{s+1} A.$$

According to the induction assumption from
$$r(X_{1} X_{2} \ldots X_{s} A_{1} )<r(X_{2} \cdots X_{s} A_{1} )<\cdots <r(X_{s} A_{1} )<r(A_{1} )$$
it follows that
$$c(X_{1} X_{2} \cdots X_{s} X_{s+1} A)=c(X_{1} X_{2} \cdots X_{s} A_{1} )<c(A_{1} )=c(X_{s+1} A)<c(A),$$
with which we have proven a).

b) is proven similarly to a).
\end{proof}

In effect is also the dual to Theorem \ref{Theorem1} statement, in which  instead of the sign $<$ everywhere we put the sign $>$.

\begin{theorem}\label{dualth} {\rm ({\bf Dual theorem})}
Let $A$ be an arbitrary matrix from ${\mathcal M}_{n\times m}^p $. Then:

a) If $X_{1} ,X_{2} ,\cdots ,X_{s} \in {\rm {\mathcal T}}_{n} $ are such that
$$r(X_{1} X_{2} \ldots X_{s} A)>r(X_{2} X_{3} \ldots X_{s} A)>\cdots >r(X_{s-1} X_s A) >r(X_{s} A)>r(A),$$
then
$$c(X_{1} X_{2} \ldots X_{s} A)>c(A).$$

b) If $Y_{1} ,Y_{2} ,\cdots ,Y_{t} \in {\rm {\mathcal T}}_{m} $ are such that
$$c(AY_{1} Y_{2} \ldots Y_{t})>c(AY_{1} Y_{2} \ldots Y_{t-1})>\cdots >c(AY_1 Y_2 )> c(AY_1 )>c(A),$$
then
$$r(AY_{1} Y_{2} \ldots Y_{t} )>r(A).$$
\end{theorem}

\section{Semi-canonical and canonical ${\mathcal M}_{n\times m}^p$-matrices}
\begin{definition}\label{Definition2}
\rm Let $A\in {\mathcal M}_{n\times m}^p$,
$r(A)=\langle x_{1} ,x_{2} ,\ldots ,x_{n} \rangle$ and
$c(A)=\langle y_{1} ,y_{2} ,\ldots ,y_{m} \rangle$.
We will call the matrix $A$ \emph{semi-canonical}, if
$$x_{1} \le x_{2} \le \cdots \le x_{n} $$
and
$$y_{1} \le y_{2} \le \cdots \le y_{m} .$$
\end{definition}





\begin{lemma}\label{Lemma1}
Let $A=\left( a_{st} \right)_{n\times m} \in {\mathcal M}_{n\times m}^p $ be a semi-canonical matrix. Then there exist integers $s,t$, such that $1\le s\le n$, $1\le t\le m$ and
\begin{equation} \label{GrindEQ__2_}
 a_{1 1} =a_{1 2} =\cdots =a_{1 s} =0,\quad 1\le a_{1, s+1} \le a_{1, s+2} \le \cdots \le a_{1 m} \le p-1, \end{equation}
\begin{equation} \label{GrindEQ__3_}
a_{1 1} =a_{2 1} =\cdots =a_{t 1} =0,\quad 1\le a_{t+1, 1} \le a_{t+2, 1} \le \cdots \le a_{n 1} \le p-1. \end{equation}
\end{lemma}

\begin{proof}
Let $r(A)=\langle x_{1} ,x_{2} ,\ldots x_{n} \rangle $ and $c(A)=\langle y_{1} ,y_{2} ,\ldots y_{m} \rangle $.
We assume that there exist integers $p$ and $q$, such that $1\le p<q\le m$, $a_{1 p} \ge a_{1 q}$.
In this case $y_{p} >y_{q} $, which contradicts the condition for semi-canonicity of the matrix $A$. We have proven (\ref{GrindEQ__2_}). Similarly, we prove (\ref{GrindEQ__3_}) as well.
\end{proof}

\begin{definition}\label{Definition4}
\rm We will call the matrix $A\in {\mathcal M}_{n\times m}^p $ \emph{canonical matrix}, if $r(A)$ is the minimal element with respect to the lexicographic order in the set $\{ r(B)\; |\; B \sim A\}$.
\end{definition}

\begin{problem}\label{probl4}
For given $m$, $n$ and $p$, find all canonical ${\mathcal M}_{n\times m}^p$-matrices satisfying  certain conditions.
\end{problem}

Particular cases of Problem \ref{probl4} are as follows:

\begin{problem}\label{probl5}
For given $n$ and $k$, find all  $n\times n$ canonical weighing matrix with weight $k$.
\end{problem}

\begin{problem}\label{probl6}
For given $n$, find all  $n\times n$ canonical Hadamard matrices.
\end{problem}

If the matrix $A\in {\mathcal M}_{n\times m}^p $ is canonical and $r(A)=\langle x_{1} ,x_{2} ,\ldots ,x_{n} \rangle ,$ then obviously

\begin{equation} \label{GrindEQ__6_} x_{1} \le x_{2} \le \cdots \le x_{n} . \end{equation}

From Definition \ref{Definition4} immediately follows that there exists only one canonical binary matrix in every class on the equivalence relation $"\sim "$ (see Definition \ref{Definition3}).

\begin{lemma}\label{Corollary2}
If the matrix $A\in {\mathcal M}_{n\times m}^p $  is a canonical matrix, then $A$ is a semi-canonical matrix.
\end{lemma}

\begin{proof}
Let $A\in {\mathcal M}_{n\times m}^p $ be a canonical matrix and $r(A)=\langle x_{1} ,x_{2} ,\ldots ,x_{n} \rangle $. Then from (\ref{GrindEQ__6_}) it follows that $x_{1} \le x_{2} \le \cdots \le x_{n} $. Let $c(A)=\langle y_{1} ,y_{2} ,\ldots ,y_{m} \rangle $. We assume that there are $s$ and $t$ such that $s\le t$ and $y_{s} >y_{t} $. Then we swap the columns of numbers $s$ and $t$. Thus we obtain the matrix $A'\in {\mathcal M}_{n\times m}^p $, $A'\ne A$. Obviously $c(A')<c(A)$. From Theorem \ref{Theorem1} and Theorem \ref{dualth} it follows that $r(A')<r(A),$ which contradicts the minimality of $r(A)$.
\end{proof}

In the next example, we will see that the opposite statement of Lemma \ref{Corollary2} is not always true.

\begin{example}\label{Example1}
\rm We consider the matrices:
$$A=\left(\begin{array}{cccc}
{0} & {0} & {1} & {2} \\
{0} & {0} & {2} & {2} \\
{0} & {2} & {0} & {0} \\
{1} & {0} & {0} & {0}
\end{array}\right) \in {\mathcal M}_{4\times 4}^3$$
and
$$B=\left(\begin{array}{cccc}
{0} & {0} & {0} & {2} \\
{0} & {1} & {2} & {0} \\
{0} & {2} & {2} & {0} \\
{1} & {0} & {0} & {0} \end{array}\right)  \in {\mathcal M}_{4\times 4}^3 .$$

After immediate verification, we find that $A\sim B$. Furthermore $r(A)=\langle 5,8,18,27\rangle $, $c(A)=\langle 1,6,45,72\rangle $, $r(B)=\langle 2,15,24,27\rangle $,
$c(B)=\langle 1,15,24,54\rangle $. So $A$ and $B$ are two equivalent semi-canonical matrices, but they are not canonical. The canonical matrix in this equivalence class is the matrix
$$C=\left(\begin{array}{cccc}
{0} & {0} & {0} & {1} \\
{0} & {0} & {2} & {0} \\
{1} & {2} & {0} & {0} \\
{2} & {2} & {0} & {0} \end{array}\right) \in {\mathcal M}_{4\times 4}^3 ,$$
where $r(C)=\langle 1,6,45,72\rangle$ and $c(C)=\langle 5,8,18,27\rangle$.

\hfill $\square$
\end{example}

From Example \ref{Example1} immediately follows that there may be more than one semi-canonical element in a given equivalence class.

\section{Necessary and sufficient conditions for a ${\mathcal M}_{n\times m}^p$-matrix to be canonical}

Let $A=(a_{ij} )\in {\mathcal M}_{n\times m}^p$, $r(A)=\langle x_1 ,x_2 , \ldots , x_n \rangle$. We introduce the following notations:

\begin{description}
\item[$\nu_i (A)$]$\displaystyle=\nu (x_i ) $ = the number of nonzero entries in the $i$-th row of $A$, $i=1,2,\ldots n$.

\item[$Z_i (A)$]$\displaystyle =Z(x_i ) =\{ x_k \in r(A) |\; x_k =x_i \}$ -- the set of all rows $x_k \in r(A)$, such that $x_k =x_i$. By definition $x_i \in Z (x_i )$, $i=1,2,\ldots n$.

\item[$\zeta_i (A)$]$\displaystyle=\zeta (x_i ) = |Z_i (A) |$, $i=1,2,\ldots n$.
\end{description}

\begin{lemma}\label{tv5}
Let $A=(a_{ij} )\in {\mathcal M}_{n\times m}^p$, $r(A)=\langle x_1 ,x_2 , \ldots , x_n \rangle$ and let $x_1 \le x_2 \le \cdots \le x_n$. Then for each $i=2,3,\ldots ,n$, for which $x_{i-1} <x_i$, or $i=1$ the condition
$$ Z(x_i )=\{ x_i ,x_{i+1} ,\ldots , x_{i+\zeta (x_i ) -1} \} $$
is fulfilled.
\end{lemma}

\begin{proof} Trivial.
\end{proof}

The formulation of the following theorem will help us to construct a recursive algorithm for obtaining all canonical ${\mathcal M}_{n\times m}^p$-matrices.

\begin{theorem}\label{mainth}
Let $A=(a_{ij} )\in {\mathcal M}_{n\times m}^p$, $r(A)=\langle x_1 ,x_2 , \ldots , x_n \rangle ,$ $c(A)=\langle y_1 ,y_2 ,\ldots ,y_m \rangle$, $s=\nu_1 (A)$, $t=\zeta_1 (A)$. Then $A$ is canonical if and only if  the following conditions are satisfied:
\begin{enumerate}
\item \label{cnd1} $x_1 \le x_2 \le \cdots \le x_n \le p^m -1$;
\item \label{cnd2} $\displaystyle \frac{p^s -1}{p-1} \le x_1 \le p^s -1$;
\item \label{cnd3} If $s>1$ then $y_{m-s+1} \le y_{m-s+2} \le y_m$;
\item \label{cnd4} For each  $i=2,3,\ldots ,n$, $\nu_1 (A) \le \nu_i (A)$;
\item \label{cnd5} Let $t<n$. Let an integer $i$ exist such that $t<i\le n$ and $\nu_i (A)=\nu_1 (A)=s$. Then  we successively get the matrices $A'$, $A''$ and $A'''$ in the following way:
    \begin{enumerate}
      \item We get the matrix $A'$ by moving the rows from the set $Z_i (A)$ so they become first;
      \item If $s=m$ then $A'' =A'$. Let $s<m$, $A' =(a_{ij}' )$ and let $\Upsilon =\{ j\; |\; a_{1\, j}' \ne 0 \} =\{ u_1 ,u_2 , \ldots u_s \}$. Then we get the matrix $A''$ by moving successively the $u_k$-th column $(k=1,2,\ldots , s)$ from $A'$ so it becomes last in $A''$;
      \item We get the matrix $A'''$ by sorting the last $s$ columns of $A''$ in ascending order.
    \end{enumerate}
Then $r(A)\le r(A''')$.
\item \label{cnd6} Let  $1\le t<n$ and $0\le s<m$. Let the matrix $B\in \mathcal{M}_{(n-t)\times (m-s)}^p$ be obtained from $A$ by removing the first $t$ rows and the last $s$ columns. Then $B$ is canonical.
\end{enumerate}
\end{theorem}

\begin{proof}

\textbf{Necessity}. Let $A=(a_{ij} )\in {\mathcal M}_{n\times m}^p$ be a canonical matrix and let $r(A)=\langle x_1 ,x_2 , \ldots , x_n \rangle$, $c(A)=\langle y_1 ,y_2 ,\ldots ,y_m \rangle$.

Conditions \ref{cnd1} follows from the fact that every canonical matrix is semi-canonical (Lemma \ref{Corollary2}), so $x_1 \le x_2 \le \cdots \le x_n$  and from inequality (\ref{p^m-1}).

From equation (\ref{x_i}) and Lemma \ref{Lemma1} it follows that
$$x_1 =\sum_{j=1}^m a_{1j} p^{m-j} = \sum_{j=m-s+1}^m a_{1j} p^{m-j} \ge \sum_{j=m-s+1}^m 1\cdot p^{m-j} =\frac{p^s -1}{p-1}$$
and
$$x_1 = \sum_{j=m-s+1}^m a_{1j} p^{m-j} \le \sum_{j=m-s+1}^m (p-1) p^{m-j} =(p-1)\frac{p^s -1}{p-1} =p^s -1 .$$
Therefore condition \ref{cnd2} is true.

Conditions \ref{cnd3} follows from the fact that every canonical matrix is semi-canonical (Lemma \ref{Corollary2}).

We assume that an integer $i$, $2\le i\le n$ exists, such that $\nu_i (A)<\nu_1 (A)=s$ and let $\nu_i (A)=u<s$. Then a matrix $A'=(a_{i\, j}')\sim A$ exists such that $a_{i\, 1}'=a_{i\, 2}' =\cdots =a_{i\, m-u}' =0 $ and $1\le a_{i\, m-u+1}' \le a_{i\, m-u+2}' \le \cdots \le a_{i\, m}' \le p-1$. We move the $i$-th row  of $A'$ at first place and we obtain a matrix $A'' $. Obviously $A'' \sim A$. Let $r(A'' )=\langle x_1'' ,x_2'' ,\ldots ,x_n'' \rangle$.
From above proven condition \ref{cnd2}, it follows that $\displaystyle x_1 \ge \frac{p^s -1}{p-1} =p^{s-1} + p^{s-2} +\cdot +p^u +p^{u-1} +\cdots + p+1>p^u >p^u -1\ge x_1''$. Therefore $x_1 >x_1''$, i.e.  $r(A) > r(A'' )$, which is impossible, due to the fact that $A$ is canonical. Thus, condition \ref{cnd4} is true.

Condition \ref{cnd5} comes directly from the fact that $A$ is canonical and $r(A)\le r(U)$ for each matrix $U\sim A$.

Let $t=\zeta_1 (A)<n$ and let $s=\nu_1 (A)<m$. From the already proved conditions \ref{cnd1}, \ref{cnd2}, \ref{cnd4} and \ref{cnd5} and Lemma \ref{tv5}, it follows that $A$ is presented in the form:
\begin{equation}\label{kiki}
A= \left(
  \begin{array}{cc}
    O & N \\
    B & C \\
  \end{array}
\right) ,
\end{equation}
where $O$ is $t\times (m-s)$ matrix, all elements of which are equal to 0, $N$ is $t\times s$ matrix, all elements of which are equal to each other and which are not equal to 0 and all rows of the matrix $(O\ N)_{t\times m}$ coincide with the elements of the set $Z_1 (A)$, $B \in \mathcal{M}_{(n-t)\times (m-s)}^p$, $C \in \mathcal{M}_{(n-t)\times s}^p$.

Let $B' \sim B$ and let $B'$ be a canonical $\mathcal{M}_{(n-t)\times (m-s)}^p$-matrix. Then the following matrices $A' \in {\mathcal M}_{n\times m}^p$ and $C'\in \mathcal{M}_{(n-t)\times s}^p$ exist, such that $A' \sim A$, $C' \sim C$, $\displaystyle A'= \left(
  \begin{array}{cc}
    O & N \\
    B' & C' \\
  \end{array}
\right) $, and $C'$ is obtained from $C$ after an eventual permutation of the rows. Let $r(A')=\langle x_1' ,x_2' ,\ldots , x_n' \rangle$. Obviously $x_i' = x_i$ for all $i=1,2,\ldots t$. Let us assume that $B' \ne B$, i.e. $r(B')<r(B)$. Let $r(B) =\langle b_{t+1} ,b_{t+2} ,\ldots , b_{n} \rangle$, $r(B') =\langle b_{t+1}' ,b_{t+2}' ,\ldots , b_{n}' \rangle$, $r(C) =\langle c_{t+1} ,c_{t+2} ,\ldots , c_{n} \rangle$, $r(C') =\langle c_{t+1}' ,c_{t+2}' ,\ldots , c_{n}' \rangle$. From assumption it follows that there exist $i\in [t+1,n]$ such that $b_{t+1}' =b_{t+1}$, $b_{t+2}' =b_{t+2} ,\ldots , b_{i-1}' =b_{i-1}$ and  $b_i' < b_i$, i.e. $b_i' +1\le b_i $. Then $x_1' =x_1 ,x_2' =x_2 ,\ldots ,x_{i-1}' =x_{i-1}$.  Since $0\le c_k <p^s$ and $0\le c_i' <p^s$,  for each $i\in [t+1, n]$, then $x_i' =b_i' p^s +c_i' \le (b_i' +1)p^s +c_i' =b_i p^s +p^s +c_i' +c_i -c_i \le b_i p^s +p^s +p^s +c_i -0 < b_i p^s +c_i$. Consequently $r(A') < r(A)$. But $A$ is canonical, i.e.  $r(A)\le r(A')$, which is a contradiction. Therefore $B'=B$  and $B$ is canonical. Thus we have proved condition \ref{cnd6}.

\textbf{Sufficiency}. Let $A\in{\mathcal M}_{n\times m}^p$ satisfy conditions \ref{cnd1} $\div$ \ref{cnd6} and hence the conditions of Lemma \ref{tv5} are fulfilled. Let $r(A)=\langle x_1 ,x_2 , \ldots , x_n \rangle$ and $c(A)=\langle y_1 ,y_2 ,\ldots ,y_m \rangle$.

If $t=n$ then $x_1 =x_2 =\cdots =x_n$ and according to condition \ref{cnd3} it is easy to see that $A$ is canonical ${\mathcal M}_{n\times m}^p$-matrix.

If $t<n$ and $s=m$ then according to condition \ref{cnd1}, Lemma \ref{tv5} and conditions \ref{cnd4} and \ref{cnd5}  it is easy to see that $A$ is canonical ${\mathcal M}_{n\times m}^p$-matrix.

Let  $1\le t<n$ and $0\le s<m$. Let $U\sim A$ and let $U$ be a canonical ${\mathcal M}_{n\times m}^p$-matrix. Since the Conditions \ref{cnd1} $\div$ \ref{cnd6} are necessary for the canonicity of a matrix, consequently $U$ also satisfies these conditions. According to condition \ref{cnd4}
\begin{equation}\label{nu1=s}
\nu_1 (U) =\nu_1 (A)=s.
\end{equation}

Thus the matrix $U$ is represented in the form (\ref{kiki}) and let

\begin{equation}\label{star2}
A= \left(
  \begin{array}{cc}
    O & N \\
    B & C \\
  \end{array}
\right)
\quad \textrm{and} \quad
U= \left(
  \begin{array}{cc}
    O' & N' \\
    B' & C' \\
  \end{array}
\right) ,
\end{equation}

Let us assume that $U$ is obtained from $A$ only by permutation of the columns. In this case obviously $\zeta_1 (U) =\zeta_1 (A)=t$, $\nu_1 (U) =\nu_1 (A)=s$, $O'=O$, $N' \sim N$, $B' \sim B$ and $C' \sim C$.

Permutation of columns which are different each other and which belong only to the set $Y_1  = \{y_1 , y_2 , \ldots , y_{m-s} \}$ without permutation of different each other  rows is impossible in accordance with condition \ref{cnd6}.

Permutation of columns which are different each other and which belong only to the set $Y_2  = \{y_{m-s+1} , y_{m-s+2} , \ldots , y_{m} \}$ without permutation of different each other  rows is impossible in accordance with condition \ref{cnd3}.

Therefore there are $k,l$ such that $1\le k\le m-s <l \le m$ and the $k$-th column has become the $l$-th or the $l$-th column has become the $k$-th. Then according to condition \ref{cnd3} and equation (\ref{kiki}) easily see that it is impossible if we did not change the places of some rows.

Therefore $U$ is obtained from $A$ by swapping some of the rows.
Without loss of generality, we can assume that $U$ is obtained from $A$ in the beginning by swapping some rows, then (if it is necessary) swapping some columns.

Permutation of rows that belong only to the set $X_1  = \{x_1 , x_2 , \ldots , x_{t} \} =Z_1 (A)$ does not change the matrix $A$ because $x_1 =x_2 =\cdots =x_t$.

Permutation of rows that belong only to the set $X_2  = \{x_{t+1} , x_{t+2} , \ldots , x_n \}$ is impossible in accordance with condition \ref{cnd6}.

Therefore, taking into account the conditions  \ref{cnd1} and \ref{cnd4} and Lemma \ref{tv5}, we conclude that we have changed the first $t=\zeta_1 (A)$ rows with another equal to each rows of the set $Z_j (A)$, $t+1\le j\le n$. After that in order to obtain a matrix of kind (\ref{kiki}), if it is necessary, we have to change the places of some columns of the matrix $A$. According to conditions \ref{cnd3} and \ref{cnd5} it follows that $r(A)\le r(U)$. But $U$ is canonical, i.e. $r(U)\le l(A)$.
Therefore, $U=A$, i.e. $A$ is canonical.
\end{proof}

\makeatletter
\renewcommand{\@biblabel}[1]{[#1]\hfill}
\makeatother

\bibliographystyle{plain}
\bibliography{modn}




\end{document}